\documentclass[11pt]{amsart}

\usepackage{multirow}
\usepackage{amssymb}
\usepackage{csvsimple}
\usepackage{siunitx}
\usepackage{graphicx,color}
\def\R{\mathbb{R}}

\def\cA{\mathcal{A}}

\def\cI{\mathcal{I}}

\def\cN{\mathcal{N}}
\def\cO{\mathcal{O}}

\def\cS{\mathcal{S}}
\def\cT{\mathcal{T}}

\def\a{\alpha}
\def\b{\beta}

\def\d{\delta}

\def\l{\lambda}

\def\p{\partial}

\def\veps{\varepsilon}

\def\vphi{\varphi}
\def\O{\Omega}

\def\G{\Gamma}

\def\wto{\rightharpoonup}

\def\tu{\widetilde{u}}

\newcommand{\dv}[1]{\,{\mathrm d}#1}
\newcommand{\dual}[3][]{#1\langle #2,#3#1\rangle}

\newcommand{\wcheck}[1]{#1\hspace{-.8ex}\mbox{\huge {\lower.45ex \hbox{$\textstyle \check{}$}}} \hspace{.5ex}}

\DeclareMathOperator{\diver}{div}

\let\oldmarginpar\marginpar
\renewcommand\marginpar[1]{
  \oldmarginpar[\raggedleft\footnotesize #1]
  {\raggedright\footnotesize #1}}
\usepackage{ifthen}
\usepackage[normalem]{ulem}

\newtheorem{definition}{Definition}

\newtheorem{proposition}[definition]{Proposition}
\newtheorem{theorem}[definition]{Theorem}
\newtheorem{corollary}[definition]{Corollary}

\newtheorem{remarks}[definition]{Remarks}
\newtheorem{example}[definition]{Example}

\newtheorem{algorithm}[definition]{Algorithm}
\numberwithin{definition}{section}
\definecolor{tourquoise}{RGB}{0,170,180}	
\definecolor{darkred}{RGB}{238,34,34}		
\definecolor{darkgreen}{RGB}{0,190,0}		
\definecolor{lightgray}{RGB}{210,210,210}	
\definecolor{deepblue}{RGB}{0,0,240}		
\definecolor{darkgray}{RGB}{144,144,144}	
\definecolor{kingblue}{RGB}{64,96,224}		
\definecolor{gold}{RGB}{240,208,0}		
\definecolor{verydarkred}{RGB}{176,0,0}		

\usepackage{boxedminipage}
\def\tu{\widetilde{u}}
\def\pO{{\partial\O}}
\def\thick{\ell}
\def\id{\mathrm{id}}

\usepackage{stmaryrd}

\begin{document}

\title[Computation of optimally insulating films]{Numerical 
solution of a nonlinear eigenvalue problem arising in optimal 
insulation}

\author{S\"oren Bartels}
\author{Giuseppe Buttazzo} 

\date{\today}



\begin{abstract}
The optimal insulation of a heat conducting body by a thin film of variable thickness can be formulated as a nondifferentiable, nonlocal eigenvalue problem. The discretization and iterative solution for the reliable computation of corresponding eigenfunctions that determine the optimal layer thickness are addressed. Corresponding numerical experiments confirm the theoretical observation that a symmetry breaking occurs for the case of small available insulation masses and provide insight in the geometry of optimal films. An experimental shape optimization indicates that convex bodies with one axis of symmetry have favorable insulation properties. 
\end{abstract}

\maketitle

\section{Introduction}\label{sintro}

Improving the mechanical properties of an elastic body by surrounding it by a thin reinforcing film of a different material is a classical and well understood problem in mathematical analysis~\cite{BrCaFr80,Frie80,AceBut86,CoKaUh99}. A recent result in~\cite{BuBuNi16} proves the surprising fact that in the case of a small amount of material for the surrounding layer, an unexpected break of symmetry occurs, i.e., a nonuniform arrangement on the surface of a ball leads to better material properties than a uniform one.

The case of the heat equation is similar, and a model reduction for thin films leads, in the long-time behavior, to a partial differential equation with Robin type boundary condition, e.g.,
\[
-\Delta u=f\text{ in }\O,\qquad \thick\,\p_n u+u=0\text{ on }\pO,
\]
i.e., the heat flux $-\p_n u$ trough the boundary is given by the temperature difference divided by the scaled nonnegative layer thickness $\thick$. A vanishing tickness thus leads to a homogenous Dirichlet boundary condition which prescribes the external temperature (set to zero), while as $\thick\to+\infty$ the boundary condition above approaches the Neumann one, corresponding to a perfect insulation. The long-time insulation properties (decay rate of the temperature) are determined by the principal eigenvalue of the corresponding differential operator, i.e., 
\[
\l_\thick=\inf\left\{\int_\O|\nabla u|^2\dv{x} + \int_\pO\thick^{-1}u^2\,\dv{s}\ :\ \int_\O u^2\dv{x}=1\right\}
\]
The optimality of the layer is characterized by minimality of 
$\l_\thick$ among admissible arrangements $\thick:\pO\to\R_+$ with prescribed total mass $m$, i.e., $\|\thick\|_{L^1(\pO)} = m$. Interchanging the minimization with respect to $\thick$ and $u$ leads to an explicit formula for the optimal $\thick$, through the nonlinear eigenvalue problem 
\[\begin{split}
\l_m&=\inf\left\{\l_\thick\ :\ \int_\pO\thick\,\dv{s}=m\right\}\\
&=\inf\left\{\int_\O|\nabla u|^2\dv{x} + \frac1m\Big(\int_\pO|u|\,\dv{s}\Big)^2\ :\ \int_\O u^2\dv{x}=1\right\}.
\end{split}\]
The calculation shows that 
optimal layers $\thick$ are proportional to traces of nonnegative
eigenfunctions $u_m$ on $\pO$. 

The results in~\cite{BuBuNi16} prove in a nonconstructive way that 
for the unit ball nonradial eigenfunctions exist if and only if
$0<m<m_0$, where~$m_0$ is a critical mass corresponding to the first nontrivial Neumann eigenvalue of the Laplace operator. In particular, 
the symmetry breaking occurs if and only if $\l_N<\l_m<\l_D$, where $\l_N$ and $\l_D$ denote the first (nontrivial) eigenvalues of the Laplacian with Neumann and Dirichlet boundary conditions.

While the proof of the result above implies that optimal nonradial
insulating films have to leave gaps (i.e. regions on $\pO$ where $\thick=0$) on the surface of the ball, the analysis does not characterize further properties such as symmetry or connectedness of the gaps. It is therefore desirable to gain insight of qualitative and quantitative features via accurate numerical
simulations.

Computing solutions for the nonlinear eigenvalue problem 
defining $\l_m$ is a challenging task since this requires solving a 
nondifferentiable, nonlocal, constrained minimization problem. 
To cope with these difficulties we adopt a gradient flow approach
of a suitable regularization of the minimization problem, i.e., 
we consider the evolution problem 
\[
(\p_t u, v) + (\nabla u,\nabla v) 
+ \frac1m \Big(\int_\pO |u|_\veps \dv{s} \Big) \int_\pO \frac{uv}{|u|_\veps} \dv{s}
= 0. 
\]
Here, $(\cdot,\cdot)$ denotes the $L^2$ inner product on $\O$. 
The regularized modulus is defined via
$|a|_\veps = (a^2+\veps^2)^{1/2}$ and the constraint
$\|u\|_{L^2}=1$ is incorporated via the conditions
\[
(\p_t u,u) = 0, \qquad (v,u) = 0.
\]
With the backward difference quotient $d_t u^k = (u^k-u^{k-1})/\tau$
for a step size $\tau>0$ we consider a semi-implicit discretization
defined by the sequence of problems that determines the sequence
$(u^k)_{k=0,1,\dots}$ for a given initial~$u^0$ recursively via
\[
(d_t u^k,v) + (\nabla u^k,\nabla v) + m^{-1} \Big(\int_\pO |u^{k-1}|_\veps \dv{s}\Big)
\int_\pO \frac{u^kv}{|u^{k-1}|_\veps} \dv{s} = 0,
\]
for all test functions~$v$ subject to the constraints
\[
(d_t u^k,u^{k-1}) = 0, \qquad (v,u^{k-1}) = 0.
\]
Note that every step only requires the solution of a constrained linear
elliptic problem. Crucial for this is the semi-implicit treatment of
the nondifferentiable and nonlocal boundary term and the normalization
constraint. We show that this time-stepping scheme is nearly
unconditionally energy decreasing in terms of $\tau$ and $\veps$
and that the constraint is approximated appropriately. The stability analysis
is related to estimates for numerical schemes for mean curvature and total 
variation flows investigated in~\cite{Dziu99,BaDiNo17-pre}. 

The spatial discretization of the minimization problem and the iterative
scheme require an appropriate numerical integration of the boundary 
terms. We provide a full error analysis for the use of a straightforward
trapezoidal rule avoiding unjustified regularity assumptions. This 
leads to a convergence rate for the approximation of $\l_m$ incorporating
both the mesh size $h>0$ and the regularization parameter $\veps>0$.
The good stability properties of the discrete gradient flow and
the accuracy of the spatial discretization are illustrated by means
of numerical experiments. These reveal that for moderate triangulations
with a few thousand elements a small number of iterations is sufficient
to capture the characteristic properties of solutions of the nonlinear
eigenvalue problem and thereby gain understanding in the features of
optimal nonsymmetric insulating films for the unit ball in two and 
three space dimensions. 

We also investigate the idea of improving the insulation properties by modifying the shape of a heat conducting body. In the case of two space dimensions we use a shape derivative and deform a given domain via a negative shape gradient obtained via appropriate Stokes problems. Corresponding numerical experiments confirm the observation
from~\cite{BuBuNi16} that the disk is not optimal when the total amount $m$ of insulating material is small and that instead convex domains with one axis of symmetry lead to smaller principal eigenvalues. 

In three or more space dimensions the situation is more complex: indeed an optimal shape does not exist. In fact, if $\O$ is composed of a large number $n$ of small disjoint balls of radius $r_n\to0$ we may define
\[u=
\begin{cases}
1&\hbox{on one of the balls}\\
0&\hbox{on the remaining ones}
\end{cases}\]
and we obtain, if $B$ is the ball where $u=1$,
\[
\l_m(\O)\le\frac{\frac1m(|\partial B|)^2}{|B|}=\frac{d^2\omega_d}{m}r_n^{d-2}\;,
\]
where $\omega_d$ denotes the Lebesgue measure of the unit ball in $\R^d$. If $d\ge3$ we then obtain that $\l_m(\O)$ may be arbitrarily close to zero. Nevertheless, starting with a ball as the initial domain and performing a shape variation among rotational bodies, we numerically identify ellipsoids and egg-shaped domains that have good insulation properties.

In spite of the nonexistence argument above, it is desirable to prove (or disprove) that an optimal domain exists in a restricted class, as for instance the class of convex domains. The numerical experiments of Section~\ref{sec:shape_opt} indicate that convex bodies are optimal among rotational bodies, which is a good sign for the existence of an optimal body among convex ones.

The article is organized as follows. In Section~\ref{sec:eigval_prob} 
we outline the derivation of the nonlinear eigenvalue problem. In 
Section~\ref{sec:grad_flow} we investigate the stability properties
of the semi-implicit discretization of the gradient flow used 
as an iterative scheme for computing eigenfunctions. 
Section~\ref{sec:spatial} is devoted to the analysis of the spatial
discretization of the problem. Experiments confirming the stability
and approximation properties and revealing the qualitative and
quantitative properties of optimal insulating films are presented
in Section~\ref{sec:num_ex}. In Section~\ref{sec:shape_opt} we
experimentally investigate the numerical optimization of the
shape of insulated conducting bodies.

\section{Nonlinear eigenvalue problem}\label{sec:eigval_prob}

We consider a heat conducting body $\O\subset\R^d$ that is surrounded by an insulating material of variable normal thickness $\veps\thick \ge 0$, cf.~Fig.~\ref{fig:insulating_body}. A model reduction for vanishing conductivity $\veps\to0$ in the insulating layer leads, for the stationary temperature $u$ under the action of heat sources $f$, to an elliptic partial differential equation, with Robin type boundary condition, i.e.,
\[
-\Delta u = f \mbox{ in }\O, \qquad 
\thick \, \p_n u + u = 0 \mbox{ on }\pO,
\]
cf.~\cite{BrCaFr80,AceBut86} for details. The boundary condition states that the heat flux through the boundary is given by the temperature difference divided by thickness of the insulating material. 

\begin{figure}[h]
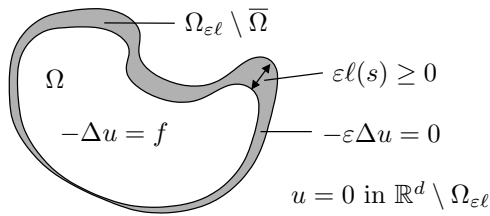
\caption{Conducting body surrounded by insulating material of variable normal thickness $\veps\thick\ge0$.}
\label{fig:insulating_body}
\end{figure}

It has to be noticed that the only interesting case occurs when the conductivity and the thickness of the insulating material have the same order of magnitude. Indeed, if conductivity is significantly smaller than thickness the limit problem is the Neumann one, while in the converse situation one obtains the Dirichlet problem.

The optimization of the thickness of the thin insulating layer, once the total amount of insulator $\|\thick\|_{L^1}$ is prescribed, is illustrated in \cite{BuBuNi16}. Here we deal with the case when no heat source is present, so that the temperature $u$, starting from its initial datum, tends to zero as $t\to+\infty$. It is well known that the temperature decays exponentially in time at a rate given by the first eigenvalue $\l_\thick$ of the differential operator $\cA_\thick$ defined by
\[
\dual{\cA_\thick u}{v} = \int_\O \nabla u \cdot \nabla v\,\dv{x}
 + \int_\pO \thick^{-1} u v\,\dv{s},\qquad u,v\in H^1(\O).
\]
We have then
\[
\l_\thick = \min\left\{\dual{\cA_\thick u}{u}\ :\ \int_\O|u|^2\dv{x}=1\right\}
\]
and, if we look for the distribution of insulator around $\O$ which gives the slowest decay in time, we have to solve the optimization problem
\[
\min\left\{\l_\thick\ :\ \int_\pO\thick\dv{s} = m\right\},
\]
where $m$ represents the total amount of insulator at our disposal.

Since $\l_\thick$ is given by a minimum too, we may interchange the minima over $\thick$ and $u$ obtaining that for a given $u\in H^1(\O)$ the optimal $\thick$ is such that $u^2/\thick^2$ is constant on $\pO$ and hence given by
\[
\thick(z) = \frac{m |u(z)|}{\int_\pO|u|\,\dv{s}}\qquad\hbox{for }z\in\pO.
\]
An optimal thickness $\thick$ is thus directly obtained from a solution
of the nonlinear eigenvalue problem
\[
\l_m=\min\left\{J_m(u)\ :\ \int_\O|u|^2\dv{x}=1\right\}
\]
where $J_m$ is the functional defined on $H^1(\O)$ by
\[
J_m(u)=\int_\O|\nabla u|^2\dv{x} + \frac1m\Big(\int_\pO|u|\,\dv{s}\Big)^2.
\]
The mapping $m\mapsto\l_m$ is a continuous and strictly decreasing function with the asymptotic values
\[
\lim_{m\to0}\l_m=\l_D,\qquad\lim_{m\to\infty}\l_m=0,
\]
which represent the first Dirichlet and Neumann eigenvalues of the Laplacian.

When $\O$ is a ball, denoting by
$$\l_N=\min\left\{\int_\O|\nabla u|^2\dv{x}\ :\ \int_\O u^2\dv{x}=1,\ \int_\O u\dv{x}=0\right\}$$
the first nontrivial Neumann eigenvalue, we have $0<\l_N<\l_D$ 
and there exists $m_0>0$ such that (cf.~Fig.~\ref{fig:plotlambdam})
$\l_{m_0} = \l_N$. 

\begin{figure}[h]
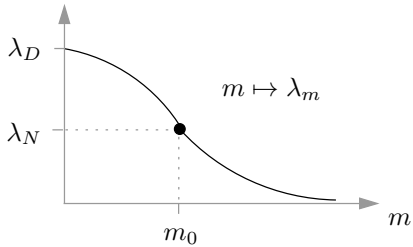
\caption{\label{fig:plotlambdam} Dependence of the principal eigenvalue 
$\l_m$ of $\cA_\thick$ on available mass $m>0$ in the case of a ball. }
\end{figure}

We summarize below the main result concerning this break of symmetry and 
refer to~\cite{BuBuNi16} for details. 

\begin{theorem}[\cite{BuBuNi16}]
Let $\O$ be a ball. If $m>m_0$ every solution $u_m$ of the minimization problem defining $\l_m$ is radial, hence the optimal thickness $\thick$ of the insulating film around $\O$ is constant. On the contrary, if $0<m<m_0$ the solution $u_m$ is not radial and so the optimal thickness $\thick$ is not constant.
\end{theorem}

The proof of the theorem also shows that nonuniform optimal insulations $\ell$
leave gaps, i.e., the support of an optimal $\ell$ is a strict subset of $\p\O$.  


\section{Iterative minimization}\label{sec:grad_flow}

We aim at iteratively minimizing the regularized functional
\[
J_{m,\veps}(u) = \|\nabla u\|^2 + \frac1m \|u\|_{L^1_\veps(\pO)}^2
\]
among functions $u\in H^1(\O)$ with $\|u\|^2=1$ and with the 
regularized $L^1$ norm defined via the regularized modulus
\[
|a|_\veps = (a^2+\veps^2)^{1/2}.
\]
Minimizers satisfy the eigenvalue equation 
\[
(\nabla u,\nabla v) 
+ \frac1m \|u\|_{L^1_\veps(\pO)} \int_\pO \frac{u v}{|u|_\veps} \dv{s} = \l_m (u,v)
\]
for all $v\in H^1(\O)$. To define an iterative scheme we consider
the corresponding evolution problem which seeks for given $u_0\in H^1(\O)$
a family $u:[0,T] \to H^1(\O)$ with $u(0)=u_0$, $\|u(t)\|^2 = 1$ for
all $t\in [0,T]$, and
\[
(\p_t u,v)_* + (\nabla u,\nabla v) 
+ \frac1m \|u\|_{L^1_\veps(\pO)} \int_\pO \frac{u v}{|u|_\veps} \dv{s} = \l_m (u,v)
\]
for all $t\in (0,T]$ and $v\in H^1(\O)$. Here, $(\cdot,\cdot)_*$ is an 
appropriate inner product defined on $H^1(\O)$. Noting that $(\p_t u, u)=0$
it suffices to restrict to test functions $v\in H^1(\O)$ with 
$(u,v)=0$ so that the right-hand side with the unknown multiplier $\l_m$ 
disappears. Replacing the time derivative by a backward difference quotient
and discretizing the nonlinear boundary term and the constraint
$(\p_t u,u)=0$ semi-implicitly to obtain
linear problems in the time steps leads to the following numerical scheme.

\begin{algorithm} \label{alg:semiimpl}
Let $\veps,\tau>0$ and $u_0\in H^1(\O)$
with $\|u_0\|^2=1$; set $k=1$. \\
(1) Compute $u^k \in H^1(\O)$ such that $(d_t u^k,u^{k-1}) = 0$ and 
\[
(d_t u^k,v)_* + (\nabla u^k,\nabla v) 
+ \frac1m \|u^{k-1}\|_{L^1_\veps(\pO)} 
\int_\pO \frac{u^k v}{|u^{k-1}|_\veps} \dv{s} = 0
\]
for all $v\in H^1(\O)$ with $(v,u^{k-1})=0$. \\
(2) Stop if $\|d_t u^k\|_*\le \veps_{\rm stop}$;
increase $k\to k+1$ and continue with~(1) otherwise. 
\end{algorithm}

The iterates of Algorithm~\ref{alg:semiimpl} approximate the continuous
evolution equation
and satisfy an approximate energy estimate on finite intervals $[0,T]$. 

\begin{proposition}\label{prop:stability}
Assume that the induced norm $\|\cdot\|_*$ on $H^1(\O)$ is such that 
we have the trace inequality
\[
(1+\veps) \|v\|_{L^1(\pO)}^2 \le c_{\rm Tr}^2 \|v\|_*^2 + \|\nabla v\|^2
\]
for some constant $c_{\rm Tr}>0$ and all $v \in H^1(\O)$. Then 
Algorithm~\ref{alg:semiimpl} is energy-decreasing in the sense
that for every $K=0,1,\dots,\lfloor T/\tau \rfloor$ we have 
\[
J_{m,\veps}(u^K) + 
2 \Big(1-\frac{c_{\rm Tr}^2 \tau}{2m}\Big) \tau \sum_{k=1}^K \|d_t u^k\|_*^2
\le J_{m,\veps}(u^0) + \frac{\veps}{m} T (1+\veps)|\p\O|^2. 
\]
Moreover, if $\|u^0\|^2 = 1$ we have that
\[
\|u^K\|^2 = 1 + \tau^2 \sum_{k=1}^K \|d_t u^k\|^2,
\]
i.e., $\|u^K\|\ge 1$ and if $\|v\|\le c_* \|v\|_*$ for all $v\in H^1(\O)$ then 
$\big|\|u^K\|^2-1\big| \le c_{\rm U} \tau$. 
\end{proposition}

\begin{proof}
Choosing $v=d_t u^k$ in Step~(1) of Algorithm~\ref{alg:semiimpl} shows that
\[\begin{split}
\|d_t u^k\|_*^2
& + \frac12 d_t \|\nabla u^k\|^2 + \frac{\tau}{2} \|\nabla d_t u^k\|^2\\
& + \frac1m \|u^{k-1}\|_{L^1_\veps(\pO)}
\int_\pO \frac{(1/2) \big(d_t |u^k|^2 + \tau |d_t u^k|^2\big)}{|u^{k-1}|_\veps}
\dv{s} = 0.
\end{split}\]
We expect the last term on the left-hand side to be related to
a discrete time derivative of the square of the regularized $L^1$ norm on 
the boundary. To verify this, we note that we have
\[\begin{split}
d_t |u^k|_\veps &= d_t \frac{|u^k|_\veps^2}{|u^k|_\veps} 
= \frac{d_t |u^k|_\veps^2}{|u^{k-1}|_\veps} 
- \frac{|u^k|_\veps d_t |u^k|_\veps}{|u^{k-1}|_\veps} 
= \frac{(1/2) \big(d_t |u^k|_\veps^2 - \tau |d_t|u^k|_\veps|^2\big)}{|u^{k-1}|_\veps}.
\end{split}\]
Using that $d_t |u^k|_\veps^2 = d_t |u^k|^2$ we combine the two identities
to verify that
\[\begin{split}
& \|d_t u^k\|_*^2
+ \frac12 d_t \|\nabla u^k\|^2 + \frac{\tau}{2} \|\nabla d_t u^k\|^2 
+ \frac1m \|u^{k-1}\|_{L^1_\veps(\pO)} d_t \|u^k\|_{L^1_\veps(\pO)} \\
& + \frac{\tau}{2m} \|u^{k-1}\|_{L^1_\veps(\pO)} 
\int_\pO \frac{|d_t|u^k|_\veps|^2 + |d_t u^k|^2}{|u^{k-1}|_\veps} \dv{s} = 0.
\end{split}\]
Note that the last term on the left-hand side is non-negative. 
We use that 
\[
\|u^{k-1}\|_{L^1_\veps(\pO)} d_t \|u^k\|_{L^1_\veps(\pO)} 
= \frac12 d_t \|u^k\|_{L^1_\veps(\pO)}^2 
- \frac{\tau}{2} \|d_t u^k\|_{L^1_\veps(\pO)}^2
\] 
to deduce 
\[\begin{split}
\|d_t u^k\|_*^2 & 
+ \frac12 d_t \big(\|\nabla u^k\|^2 + \frac{1}{m} \|u^k\|_{L^1_\veps(\pO)}^2\big) 
+ \frac{\tau}{2} \|\nabla d_t u^k\|^2 
\le \frac{\tau}{2m} \|d_t u^k\|_{L^1_\veps(\pO)}^2 \\
& \le (1+\veps) \frac{\tau}{2m} \|d_t u^k\|_{L^1(\pO)}^2 
+ \frac{\tau}{2m}\big( \veps + \veps^2 \big)|\pO|^2.
\end{split}\]
The condition that 
\[
(1+\veps)\|v\|_{L^1(\pO)}^2 \le c_{\rm Tr}^2 \|v\|_*^2 + \|\nabla v\|^2
\]
and a summation over $k=1,2,\dots,K$ imply the stability estimate. 
We further note that due to the orthogonality $(d_t u^k,u^k)=0$ we have
\[
\|u^k\|^2 = \|u^{k-1}\|^2 + \tau^2 \|d_t u^k\|^2 
\]
and an inductive argument yields the second asserted estimate.
\end{proof}

\begin{remarks}
(i) If $\|\cdot\|_*$ is the norm in $H^1(\O)$ then the continuity of
the trace operator implies the assumption. In case of the $L^2$ norm
it depends on the geometry of $\O$ via the operator norm of the trace
operator. \\
(ii) The iterates of Algorithm~\ref{alg:semiimpl} approximate an 
eigenfunction and since the values of $J_m$ are (nearly) decreasing we
expect to approximate $\l_m$ since other eigenvalues
correspond to unstable stationary configurations. \\
(iii) Since $\|u^K\|\ge 1$ we may normalize $u^K$ and obtain an approximation
of $\l_m$ via $J_m(\tu^K)$ with $\tu^K = u^K/\|u^K\|$ even if $\tau = \cO(1)$.
\end{remarks}

\section{Spatial discretization}\label{sec:spatial}

Given a regular triangulation $\cT_h$ of $\O$ with maximal mesh-size $h>0$
we consider the minimization
of $J_{m,\veps}$ in the finite element space
\[
\cS^1(\cT_h) = \big\{v_h\in C(\overline{\O}): v_h|_T \in P_1(T) 
\mbox{ for all } T\in \cT_h\big\}.
\]
For a direct implementability we include quadrature by considering 
the functional
\[
J_{m,\veps,h}(u_h) = \|\nabla u_h\|^2 + \frac1m \|u_h\|_{L^1_{\veps,h}(\pO)}^2
\]
with the discretized and regularized $L^1$ norm
\[
\|u_h\|_{L^1_{\veps,h}(\pO)} = \int_\pO \cI_h |u_h|_\veps \dv{s}
= \sum_{z\in \cN_h \cap \pO} \b_z |u_h(z)|_\veps.
\]
Here, $\cI_h:C(\overline{\O})\to \cS^1(\cT_h)$ is the nodal
interpolation operator and 
$\cN_h$ is the set of nodes in $\cT_h$ so that we have
\[
\b_z = \int_\pO \vphi_z \dv{s}
\]
for the nodal basis function $\vphi_z$ associated with $z\in \cN_h$. 
It is a straightforward task to show that the result of 
Proposition~\ref{prop:stability} carries over to the spatially
discretized functional $J_{m,\veps,h}$. The following proposition
determines the approximation properties of the discretized
functional.

\begin{proposition}\label{prop:spatial_appr}
Assume that there exists a minimizer $u\in H^2(\O)$ for $J_m$ with
$\|u\|^2=1$. We then have 
\[
0\le \min_{u_h\in \cS^1(\cT_h)} J_m(u_h) - J_m(u) \le c h \|u\|_{H^2(\O)}^2.
\]
Moreover, if $\cT_h$ is quasiuniform then for every $u_h\in \cS^1(\cT_h)$ 
we have with $\a\ge 1/2$
\[
\big| J_{m,\veps,h}(u_h) - J_m(u_h) \big| \le c (\|u_h\|_{H^1(\O)} +1)^2
(\veps + h^\a). 
\]
If $u_h$ is uniformly $H^1$-regular on the boundary, i.e., if 
$\|\nabla u_h \|_{L^2(\pO)} \le c$ for all $h>0$, then we have $\a\ge 1$.
\end{proposition}

\begin{proof}
(i) We define $\tu_h = \cI_h u / \|\cI_h u\|$ 
which is well defined for $h$ sufficiently small since 
$\|u - \cI_h u \| = \cO(h^2)$. We have that
\[
\|u-\tu_h\| + h \|\nabla [u-\tu_h]\| \le c_\cI h^2 \|D^2 u \|. 
\]
With the continuity properties of the trace operator we deduce that 
\[\begin{split}
& J_m(\tu_h) - J_m(u) \\
& \le\big(\nabla [\tu_h+u],\nabla [\tu_h-u]\big) 
 + \frac1m \big(\|\tu_h\|_{L^1(\pO)} + \|u\|_{L^1(\pO)}\big)
\|\tu_h - u \|_{L^1(\pO)} \\
&\le c h \|u\|_{H^2(\O)}^2,
\end{split}\]
which implies the first estimate. \\
(ii) Noting that $0 \le |a|_\veps - |a| \le \veps $ it follows that
\[\begin{split}
J_{m,\veps}(u_h) - J_m(u_h) 
&= \frac1m \big(\|u_h\|_{L^1_\veps(\pO)} + \|u_h\|_{L^1(\pO)} \big)
\big(\|u_h\|_{L^1_\veps(\pO)} - \|u_h\|_{L^1(\pO)}\big) \\
&\le c (\|u_h\|_{H^1(\O)} + 1) \veps . 
\end{split}\]
We further note that we have
\[
\int_\pO \big| |u_h|_\veps - \cI_h |u_h|_\veps \big|\dv{s} 
\le c h \|\nabla |u_h|_\veps\|_{L^2(\pO)}
\le c h \|\nabla u_h\|_{L^2(\pO)}.
\]
Using that for elementwise polynomial functions $\phi_h \in L^\infty(\O)$
we have
\[
\|\phi_h\|_{L^2(\pO)}\le c h^{-1/2} \|\phi_h \|_{L^2(\O)}
\]
implies that 
\[\begin{split}
\big|J_{m,\veps,h}& (u_h) - J_{m,\veps}(u_h)\big| \\
&= \frac1m \big(\|u_h\|_{L^1_{\veps,h}(\pO)} + \|u_h\|_{L^1_\veps(\pO)}\big)
\big|\|u_h\|_{L^1_{\veps,h}(\pO)}- \|u_h\|_{L^1_\veps(\pO)}\big| \\
&\le c \big(\|\nabla u_h\|+1\big) h^{1/2} \|\nabla u_h\|,
\end{split}\]
which proves the second estimate.
\end{proof}

The estimates of the proposition imply the $\Gamma$-convergence of 
the functionals $J_{m,\veps,h}$ to $J_m$ as $(\veps,h)\to 0$. We 
formally extend the discrete functionals $J_{m,\veps,h}$ by the 
value $+\infty$ in $H^1(\O)\setminus \cS^1(\cT_h)$. Note that the
functional $J_m$ is continuous, convex, and coercive
on $H^1(\O)$. 

\begin{corollary} 
The functionals $J_{m,\veps,h}$ converge to $J_m$ as $(\veps,h)\to 0$
in the sense of $\G$-convergence with respect to weak convergence in $H^1(\O)$. 
\end{corollary}

\begin{proof}
(i) Let $(u_h)_{h>0}$ be a sequence of finite element functions with
$u_h \wto u$ in $H^1(\O)$. The weak lower semicontinuity of $J_m$ yields that
\[
J_m(u) \le \liminf_{h\to 0} J_m(u_h).
\]
Since by the second estimate of Proposition~\ref{prop:spatial_appr} we have
\[
J_{m,\veps,h}(u_h) - J_m(u_h) \to 0
\]
as $(\veps,h)\to 0$, we deduce that
\[
J_m(u) \le \liminf_{h\to 0} J_{m,\veps,h}(u_h). 
\]
(ii) Let $u\in H^1(\O)$ and $\d>0$. By continuity of $J_m$ there exists
$h_0>0$ such 
\[
|J_m(u)-J_m(u_h)|\le \d/2
\]
for all $u_h \in \cS^1(\cT_h)$ with $\|u-u_h\|_{H^1(\O)} \le h_0$. 
By Proposition~\ref{prop:spatial_appr} we have 
\[
|J_m(u_h)-J_{m,\veps,h}(u_h)| \le \d/2
\]
for $(\veps,h)$ sufficiently small. By density of the spaces $\cS^1(\cT_h)$
in $H^1(\O)$ we may thus select a sequence $(u_h)_{h>0}$ with $u_h \to u$ in 
$H^1(\O)$ and 
\[
J_{m,\veps,h}(u_h) \to J_m(u)
\]
as $(\veps,h)\to 0$. 
\end{proof}

\section{Numerical experiments}\label{sec:num_ex}

We illustrate the efficiency of the proposed numerical method and identify features
of optimal insulating layers via several examples.

\begin{example}[Unit disk]\label{ex:disk}
Let $d=2$, $\O=B_1(0)$, and $m=0.4$. 
\end{example}

The Neumann and Dirichlet eigenvalues for the Laplace operator
on the unit disk coincide with certain squares of roots of Bessel functions 
and are given by 
\[
\l_D \approx 5.8503, \quad \l_N \approx 3.3979.
\]
We initialized Algorithm~\ref{alg:semiimpl} with random functions
on approximate triangulations of $\O= B_1(0)$. Table~\ref{tab:exp_disk}
displays the number of nodes and triangles in $\cT_h$, the number of 
iterations $K$ needed to satisfy the stopping criterion
\[
\|d_t u_h^K\|_* \le \veps_{\rm stop},
\]
and the approximations
\[
\l_{m,\veps,h} = \frac{J_{m,\veps,h}(u_h^K)}{\|u_h^K\|^2}
\]
with the final iterate $u_h^K$. 
We used a lumped $L^2$ inner product to define the evolution metric
$(\cdot,\cdot)_*$. The regularization parameter, the step size, 
and the stopping criterion were defined via 
\[
\veps = h/10, \quad \tau = 1, \quad \veps_{\rm stop} = h/10.
\]
Plots of the numerical solutions in Example~\ref{ex:disk}
on triangulations with~1024 and~4096 triangles are shown in 
Figure~\ref{fig:ex_disk}. The break of symmetry becomes appearant 
and is stable in the sense that it does not change with the 
discretization parameters. 
From the numbers in Table~\ref{tab:exp_disk}
we see that the iteration numbers grow slower than linearly 
with the inverse of the mesh size
and that the approximations of $\l_m$ converge without a significant
preasymptotic range. For a comparison we computed the eigenvalue 
of the operator $\cA_\thick$ with constant function 
$\thick(s) = m/|\p\O|$ and obtained the value 
$\l_{m,\thick} \approx 5.095$ for $m=0.4$, i.e., the nonuniform
distribution of insulating material reduces the eigenvalue by
approximately~$0.5\%$.

\begin{figure}[htb]
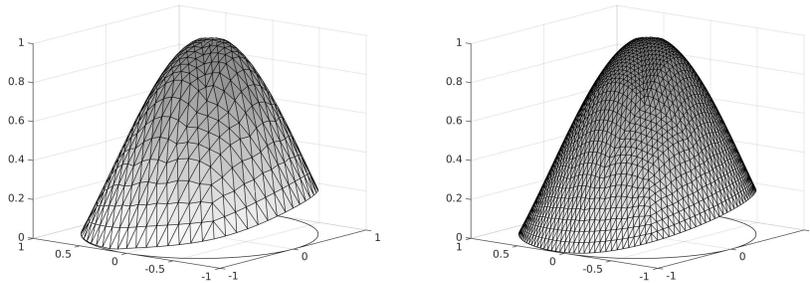

\begin{center}
\includegraphics[width=.45\linewidth]{{{disk_red_4_mass_0.4}}} 
\includegraphics[width=.45\linewidth]{{{disk_red_5_mass_0.4}}}
\end{center} \vspace*{-.5cm}
\caption{\label{fig:ex_disk}
Eigenfunctions for different triangulations in Example~\ref{ex:disk}.} 
\end{figure}

\sisetup{round-precision=6,round-mode=places,scientific-notation=true}
\begin{table}[htb]
\begin{filecontents*}{exp_disk.csv}
1,13,16,9,4.93837063
2,41,64,9,5.01813879
3,145,256,8,5.06325077
4,545,1024,16,5.07012353
5,2113,4096,70,5.07148833
6,8321,16384,107,5.07164264
7,33025,65536,161,5.07172485
8,131585,262144,224,5.07172130
9,525313,1048576,260,5.07170940
\end{filecontents*}
\csvreader[tabular=c|r|r|r|r,head=false,
 table head= $h\quad$ & $\# \cN_h$ & $\# \cT_h$ & $K$ & $\l_{m,\veps,h}$ \\\hline,late after line=\\]
 {exp_disk.csv}{}
 {$2^{-\csvcoli}$ & \csvcolii & \csvcoliii & \csvcoliv & \num{\csvcolv}}
\vspace*{2mm}
\caption{\label{tab:exp_disk} Iteration numbers and discrete eigenvalues
in Example~\ref{ex:disk}.}\vspace*{-5mm}
\end{table}

\begin{example}[Unit square]\label{ex:square}
Let $d=2$, $\O= (0,1)^2$, and $m=0.1$.
\end{example}

On the unit square we have
\[
\l_D = 2\pi^2, \quad \l_N = \pi^2,
\]
and the qualitative properties of optimal insulations differ significantly
from those for the unit disk. Figure~\ref{fig:ex_square} 
displays numerical solutions on triangulations of $\O=(0,1)^2$
with 289 and 1089 triangles in Example~\ref{ex:square}. We observe
that the computed solutions reflect the symmetry properties of
the domain but also correspond to a nonuniform distribution of
insulating material. In this experiment significantly smaller 
iteration numbers are observed which appear to be related to 
the symmetry and corresponding uniqueness properties of solutions. The fact that the computed eigenvalues shown in Table~\ref{tab:exp_square}
may increase for enlarged spaces is due to the use of the mesh-dependent 
regularization and quadrature. 

\begin{figure}[htb]
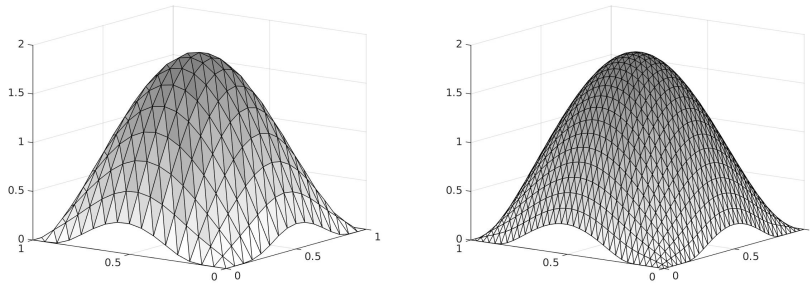

\begin{center}
\includegraphics[width=.45\linewidth]{{{square_red_4_mass_0.1}}} 
\includegraphics[width=.45\linewidth]{{{square_red_5_mass_0.1}}} 
\end{center} \vspace*{-.5cm}
\caption{\label{fig:ex_square}
Eigenfunctions for different triangulations in Example~\ref{ex:square}.} 
\end{figure}

\sisetup{round-precision=6,round-mode=places,scientific-notation=false}
\begin{table}[htb]
\begin{filecontents*}{exp_square.csv}
1,9,8,6,14.64810989
2,25,32,6,16.73424639
3,81,128,7,17.06045622
4,289,512,9,17.09022166
5,1089,2048,14,17.09352138
6,4225,8192,13,17.09444134
7,16641,32768,16,17.09198963
8,66049,131072,21,17.09009311
9,263169,524288,25,17.08932156
\end{filecontents*}
\csvreader[tabular=c|r|r|r|r,head=false,
 table head= $h\quad$ & $\# \cN_h$ & $\# \cT_h$ & $K$ & $\l_{m,\veps,h}$ \\\hline,late after line=\\]
 {exp_square.csv}{}
 {$2^{-\csvcoli}$ & \csvcolii & \csvcoliii & \csvcoliv & \num{\csvcolv}}
\vspace*{2mm}
\caption{\label{tab:exp_square} Iteration numbers and discrete eigenvalues
in Example~\ref{ex:square}.}\vspace*{-5mm}
\end{table}

\begin{example}[Unit ball]\label{ex:ball}
Let $d=3$, $\O= B_1(0)$, and $m=5.0$.
\end{example}

The effect of a nonuniform insulating layer is slightly stronger in 
three-dimensional
situations. For the setting of Example~\ref{ex:ball} and two different
triangulations we obtained the
distributions shown in Figure~\ref{fig:ex_ball}. As in two space dimensions
the insulation leaves a connected gap which is here approximately circular.
The thickness continuously increases to a maximal value that is attained 
at a point on the boundary which is opposite to the gap of the insulation. 
Note that the position of the gap is arbitrary and 
depends on the initial data and the discretization parameters. For a
uniform distribution of the insulation material we obtain the Robin
eigenvalue $\l_R = 4.7424$ so that the nonuniform distribution reduces
the slightly larger limiting value of Table~\ref{tab:exp_ball} by approximately~1\%.

\begin{figure}[htb]
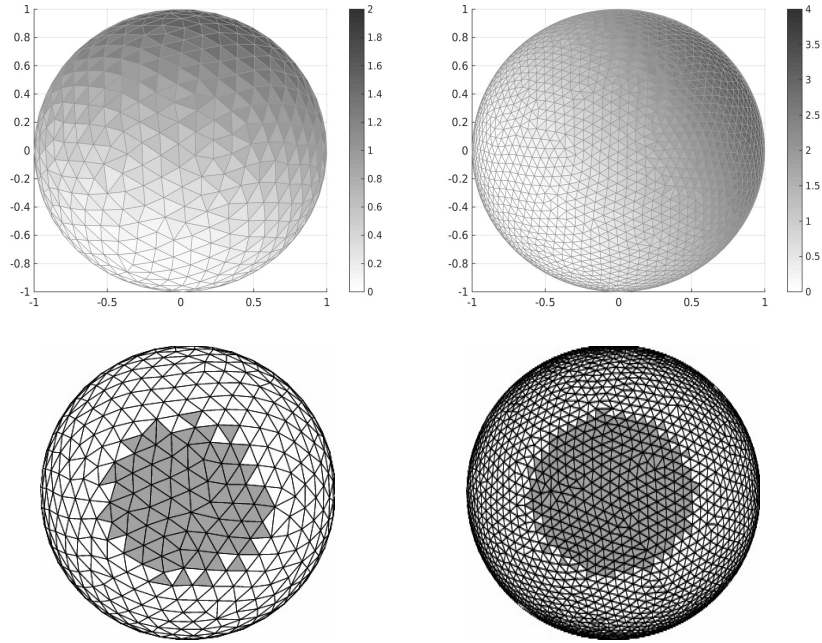

\begin{center}
\includegraphics[width=.45\linewidth,height=4.6cm]{{{ball_red_3_mass_5.0}}} 
\includegraphics[width=.45\linewidth,height=4.6cm]{{{ball_red_4_mass_5.0}}}\\[2mm]
\hspace*{-5mm}\includegraphics[width=3.9cm,height=3.9cm]{{{gap_ball_red_3_mass_5.0}}} \hspace*{1.5cm}
\includegraphics[width=3.9cm,height=3.9cm]{{{gap_ball_red_4_mass_5.0}}}
\end{center}
\caption{\label{fig:ex_ball}
Eigenfunctions for different triangulations (top) and approximate indicator 
functions of the insulation gaps after rotation (bottom) in Example~\ref{ex:ball}.} 
\end{figure}

\sisetup{round-precision=6,round-mode=places,scientific-notation=false}
\begin{table}[htb]
\begin{filecontents*}{exp_ball.csv}
1,33,87,9,4.58774288
2,257,1069,11,4.69684835
3,2205,11162,45,4.69360570
4,17461,96384,133,4.69236422
5,138745,796358,161,4.69273272
\end{filecontents*}
\csvreader[tabular=c|r|r|r|r,head=false,
 table head= $h\quad$ & $\# \cN_h$ & $\# \cT_h$ & $K$ & $\l_{m,\veps,h}$ \\\hline,late after line=\\]
 {exp_ball.csv}{}
 {$2^{-\csvcoli}$ & \csvcolii & \csvcoliii & \csvcoliv & \num{\csvcolv}}
\vspace*{2mm}
\caption{\label{tab:exp_ball} Iteration numbers and discrete eigenvalues
in Example~\ref{ex:ball}.}\vspace*{-5mm}
\end{table}

\section{Shape variations}\label{sec:shape_opt}

\subsection{Shape optimization}

The insulation properties of a conducting body can further be improved
by modifying its shape keeping its volume fixed. Taking perturbations of 
the domain gives rise
to a shape derivative of the eigenvalue $\l_m$ regarded as a function
of the domain $\O$, i.e., for a vector field $w:\O \to \R^d$ and a number
$s\in \R$ we consider the perturbed domain $(\id+s w)(\O)$ and define
\[
\d \l_m(\O)[w] = 
\lim_{s\to 0} \frac{\l_m\big((\id + sw)(\O)\big) - \l_m\big(\O\big)}{s}.
\]
It follows from, e.g.,~\cite{HenPie05-book,BuBuNi16}, that with the outer
unit normal $n$ on $\pO$ we have
\[
\d \l_m(\O)[w] = \int_\pO j_m(u) w \cdot n \dv{s},
\] 
where $j_m(u)$ is for a sufficiently regular,
nonnegative eigenfunction $u\in H^1(\O)$ 
corresponding $\l_m$ and the mean curvature $H$ on $\pO$, normalized
so that $H=d-1$ for the unit sphere, given by
\[
j_m(u) = |\nabla u|^2 - 2 |\p_n u|^2 - \l_m u^2 + \frac{2}{m} \|u\|_{L^1(\pO)}H u. 
\]
To preserve the volume of $\O$ we restrict to divergence-free vector 
fields and compute a representative $v= \nabla_{\rm St} \l_m (\O) 
\in H^1(\O;\R^d)$ of $\d\l_m(\O)$
via the Stokes problem
\[\begin{cases}
(v,w) + (\nabla v,\nabla w) + (p,\diver w) = \d\l_m(\O)[w],\\
(q,\diver v) = 0,
\end{cases}\]
for all $w\in H^1(\O;\R^d)$ and $q\in L^2(\O)$. Since $\d \l_m(\O)[w]$ only depends on
the normal component $w\cdot n$ on $\pO$ it follows that the tangential 
component of the solution
$v\in H^1(\O;\R^d)$ vanishes. To optimize $\l_m$ with respect to shape
relative to a reference domain $\O\subset \R^d$ we evolve the domain
by the negative shape gradient, i.e., beginning with $\O_0 = \O$ we
define a sequence of domains $(\O_k)_{k=0,1,\dots}$ via
\[
\O_{k+1} = (\id + \tau_k v_k) (\O_k), \quad v^k = -\nabla_{\rm St} \l_m(\O_k),
\]
with positive step sizes $\tau_k>0$. Starting
from a maximal initial step size $\tau_{\rm max}$, we decreased the step size $\tau_k$
in the $k$-th step of the gradient descent until a relative decrease of the objective
below $0<\theta<1$ is achieved. The new step size is then defined
by $\tau_{k+1} = \min\{\tau_{\rm max},2\tau_k\}$; we stop the iteration if
$\tau_{k+1}\le \veps_{\rm stop}'$. 

\begin{figure}[p]
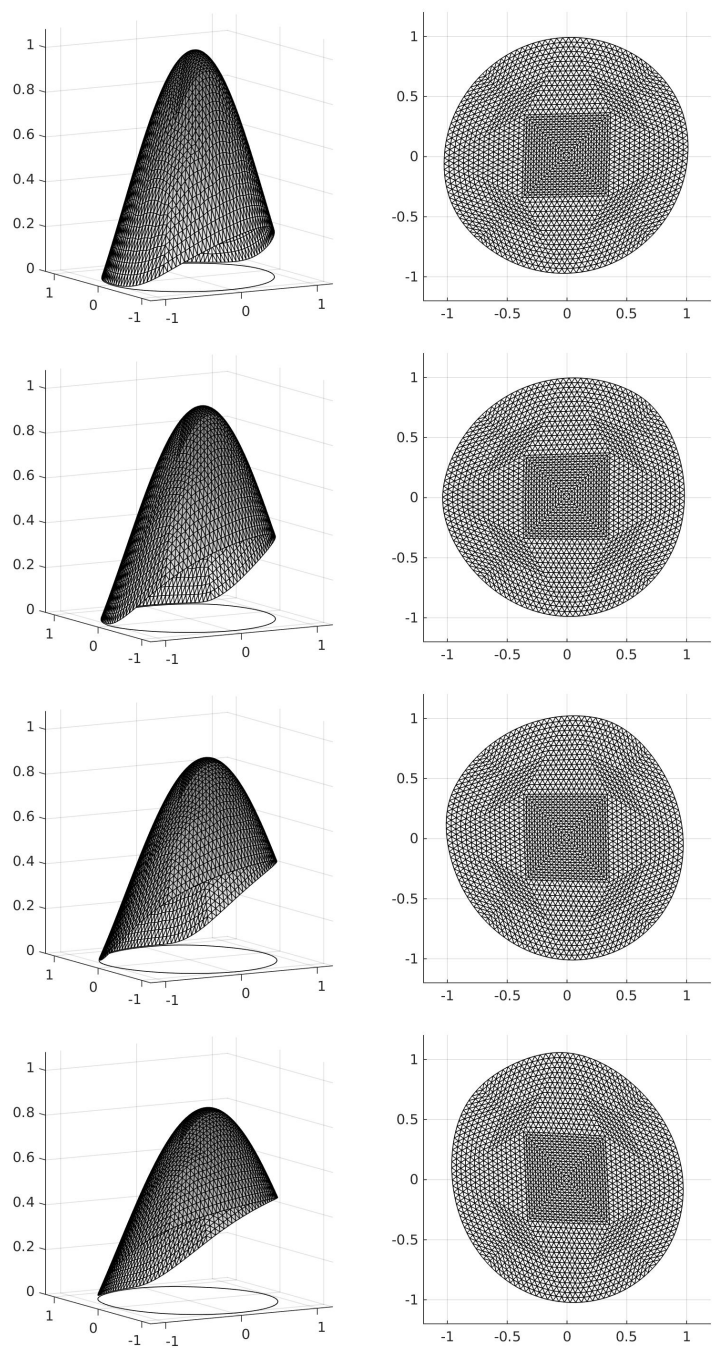

\begin{center}
\includegraphics[width=.9\linewidth,height=4.7cm]{{{shape_opt_2d_m_0.4}}} \\[-2mm]
\includegraphics[width=.9\linewidth,height=4.7cm]{{{shape_opt_2d_m_0.9}}} \\[-2mm]
\includegraphics[width=.9\linewidth,height=4.7cm]{{{shape_opt_2d_m_1.4}}} \\[-2mm]
\includegraphics[width=.9\linewidth,height=4.7cm]{{{shape_opt_2d_m_1.9}}} 
\end{center}
\caption{\label{fig:ex_opt_disk}
Eigenfunctions $u_m$ and corresponding optimized two-dimensional shapes $\O^*$ obtained from the unit disk $\O=B_1(0)$ with $m=0.4,\, 0.9,\, 1.4,\, 1.9$ (top to bottom).} 
\end{figure}

\subsection{Optimization from the unit disk}
Starting from the unit disk and using different values~$m$ for the 
available insulation mass we carried out a shape gradient descent iteration using
a discretization of the Stokes problem with a nonconforming Crouzeix--Raviart
method. The numerical solutions $v_h^k$ were projected onto conforming
$P1$ finite element vector fields before the triangulation of the current
domain was deformed. 
Figure~\ref{fig:ex_opt_disk} shows the computed nearly stationary shapes
for $m=0.4, 0.9, 1.4, 1.9$ along with eigenfunctions $u_m$. We see that
the boundary is flatter along the parts which are not insulated and
that the domains are convex with one axis of symmetry. The reduction of the eigenvalues
via shape optimization is rather small as is documented in 
Table~\ref{tab:shape_opt} in which the eigenvalues 
$\l_R^{{\rm uni}}(\O)$, $\l_{m,h}(\O)$, and $\l_{m,h}(\O^*)$
for the uniform and
nonuniform insulation of the unit disk and the optimized domains $\O^*$,
respectively, are displayed. 

\begin{table}[htb]
\begin{tabular}{l|cccc}
$m$ & 0.4 & 0.9 & 1.4 & 1.9 \\\hline
$\l_R^{{\rm uni}}(\O)$ & 5.0951 & 4.3803 & 3.8085 & 3.3519\\
$\l_{m,h}(\O)$ & 5.0714 & 4.3383 & 3.7819 & 3.3503\\
$\l_{m,h}(\O^*)$ & 5.0664 & 4.3296 & 3.7718 & 3.3378\\[2mm]
\end{tabular}
\caption{\label{tab:shape_opt} Decrease of the nonlinear eigenvalues 
via shape optimization for different total masses and comparison to uniform
insulations on the unit disk $\O=B_1(0)$.}
\end{table}

\begin{figure}[p]
\begin{center}
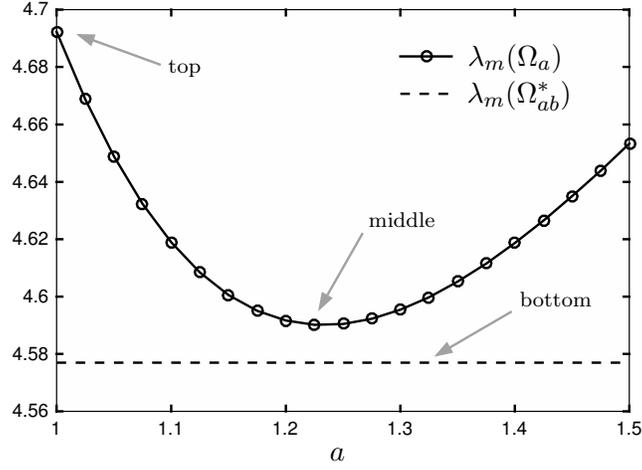
\end{center} \vspace*{-3mm}
\caption{\label{fig:ex_ellipsoids}
Eigenvalues $\l_m(\O_a)$ for different symmetric ellipsoids and optimal 
assembled half-ellipsoids $\l_m(\O_{ab}^*)$ and references to corresponding
shapes shown in Figure~\ref{fig:ex_ellipsoids_profile}.} \vspace*{-3mm}
\end{figure}
\begin{figure}[p]
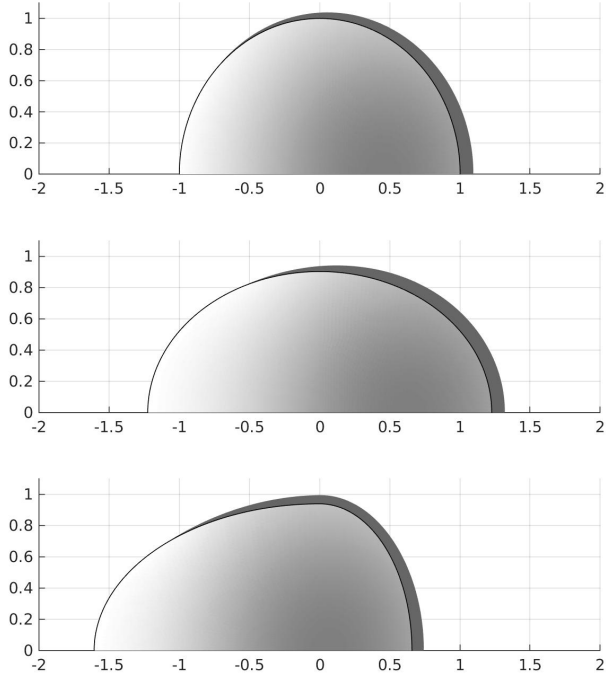

\begin{center}
\includegraphics[width=.76\linewidth]{{{optimal_sym_ellipsoid_thickness_mass_5.0_red_6}}} 
\end{center} \vspace*{-9mm}
\caption{\label{fig:ex_ellipsoids_profile}
Profile, eigenfunction (gray shading), and boundary film (scaled by $\veps = 1/10$)
for the unit ball $\O_1$ (top), 
the optimal ellipsoid $\O_a$ with $a=1.225$ (middle), and the optimal assembled half-ellipsoids
$\O^*_{ab}$ with $a = 1.607$ and $b = 0.657$ (bottom).} 
\end{figure}

\subsection{Three-dimensional rotational shapes}

The optimization of $\l_m(\O)$ among domains $\O\subset\R^d$ is ill-posed when $d\ge 3$
as explained in the introduction. 
%
The same phenomenon is obtained by taking $\O=B_{r_1}\cup B_{r_2}$ the union of two disjoint
balls with radii $r_1$ and $r_2$ and
\[u=
\begin{cases}
1&\hbox{on }B_{r_1}\\
0&\hbox{on }B_{r_2}
\end{cases}\]
which gives
\[
\l_m(\O)\le\frac{\frac1m(|\partial B_{r_1}|)^2}{|B_{r_1}|}=\frac{d^2\omega_d}{m}r_1^{d-2}\;.
\]
Again, $\l_m(\O)$ can be made arbitrarily small, keeping the measure of $\O$ 
fixed and letting $r_1\to0$.

Instabilities in 
numerical experiments confirm the general ill-posedness in the three-dimen\-si\-onal setting. It is expected that optimal shapes exist among convex bodies of fixed volume and we therefore carried out a
one-dimensional optimization among ellipsoids 
\[
\O_a = \text{ellipsoid with radii $(a,r_a,r_a)$}.
\]
The radii $r_a$ are chosen such that $|\O_a|= c_0$. 
Letting $c_0 = |B_1(0)|$ and $m=5.0$ be the volume of the unit ball
and total mass we plotted in 
Figure~\ref{fig:ex_ellipsoids} the values $\l_m(\O_a)$ as a function 
of $a\in [1,1.5]$. For numerical efficiency we exploited the rotational
symmetry of the domains and discretized the dimensionally reduced 
setting. We obtain an optimal value for the radius $a \approx 1.225$. The profile
of this ellipsoid and a corresponding eigenfunction $u_m$ are displayed
in the top and middle plot of Figure~\ref{fig:ex_ellipsoids_profile}. We further optimized the
eigenvalue within a larger class of rotational bodies defined as 
assembled half-ellipsoids, i.e., by considering 
\[
\overline{\O}_{ab} = \big(\overline{\O}_a \cap \{x_1 \le 0\}\big) \cup
\big(\overline{\O}_b \cap \{x_1 \ge 0\}\big),
\]
and adjusting the radii $r_a = r_b$ such that $|\O_{ab}| = c_0$. 
Optimizing among the radii $(a,b)$ we find the optimal shape shown 
in the bottom plot of Figure~\ref{fig:ex_ellipsoids_profile}. The corresponding 
discrete eigenfunction is visualized by the gray shading and suggests
a gap in the insulation at the pointed and a thicker insulation on 
the blunt end on the surface of the egg-like optimal domain. 

\medskip
{\em Acknowledgments:} S.B. acknowledges hospitality of the Hausdorff Research Institute for Mathematics within the trimester program {\em Multiscale Problems: Algorithms, Numerical Analysis and Computation} and support by 
the DFG via the priority program {\em Non-smooth and Complementarity-based Distributed Parameter Systems: Simulation and Hierarchical Optimization} (SPP 1962). G.B. is member of the Gruppo Nazionale per l'Analisi Matematica, la Probabilit\`a e le loro Applicazioni (GNAMPA) of the Istituto Nazio\-nale di Alta Matematica (INdAM); his work is part of the project 2015PA5MP7 {\it``Calcolo delle Variazioni''} funded by the Italian Ministry of Research and University. 

\bibliographystyle{amsalpha}
\bibliography{refs}

\bigskip
{\small\noindent
S\"oren Bartels:
Abteilung f\"ur Angewandte Mathematik, Universit\"at Freiburg\\
Hermann-Herder-Str. 10, 79104 Freiburg im Breisgau - GERMANY\\
{\tt bartels@mathematik.uni-freiburg.de}\\
{\tt https://aam.uni-freiburg.de/bartels}

\bigskip\noindent

Giuseppe Buttazzo:
Dipartimento di Matematica, Universit\`a di Pisa\\
Largo B. Pontecorvo 5, 56127 Pisa - ITALY\\
{\tt buttazzo@dm.unipi.it}\\
{\tt http://www.dm.unipi.it/pages/buttazzo/}

\end{document}